\documentclass[12pt,reqno]{amsart}

\usepackage{amsmath, amsfonts, amsthm, amssymb, color}

\newtheorem{Theorem}{Theorem}[section]
\newtheorem{Lemma}{Lemma}[section]
\newtheorem{Proposition}{Proposition}[section]

\theoremstyle{definition}
\newtheorem{Definition}{Definition}[section]

\theoremstyle{remark}
\newtheorem{Remark}{Remark}[section]

\numberwithin{equation}{section}

\renewcommand{\u}{{\bf u}}

\newcommand{\R}{{\mathbb R}}
\newcommand{\Dv}{{\rm div}}

\def\f{\frac}

\def\hf1{^\f{1}{1-\xi^2}}

\def\be{\begin{equation}}
\def\en{\end{equation}}
\def\bs{\begin{split}}
\def\es{\end{split}}

\usepackage{esint}

\newcommand{\id}{\mathrm{Id}}
\newcommand{\sym}{\mathrm{Sym}}

\newcommand{\trace}{\mathrm{trace}}

\author{Cheng Yu}
\address{Department of Mathematics,University of Florida, FL, USA.}
\email{chengyu@ufl.edu}

\title[On the Inertial Limit of Weak Solutions]
{Inertial Limit  of global weak solutions for Compressible Navier--Stokes}

\keywords{inertial limit, weak solutions, kinetic energy}
\subjclass[2000]{}

\date{\today}

\begin{document}

\maketitle

\begin{abstract}

We investigate the inertial  limit of the compressible Navier--Stokes system 
posed on the $3$-dimensional torus,
and allowing for regions of  vacuum. Considering global-in-time finite-energy weak solutions 
of a scaled system, we rigorously establish convergence to a limiting system in which the momentum 
equation reduces to a stationary elliptic balance between pressure and viscous forces. 
In this limit, the scaled kinetic energy vanishes, reflecting an overdamped regime, 
and the limiting weak solution satisfies an exact energy equality. 
Our analysis relies on uniform a priori estimates, renormalized techniques, and compactness arguments 
in the Lions--Feireisl framework, providing a mathematically rigorous analysis for the overdamped 
dynamics arising from vanishing inertia in compressible viscous flows.

\end{abstract}

\medskip

\section{Introduction and overview}

The compressible Navier--Stokes equations provide a fundamental model for viscous compressible flows, incorporating both inertial transport and dissipative mechanisms.
Depending on the relative magnitude of inertia, viscosity, and pressure forces, different asymptotic regimes arise, leading to reduced models with distinct mathematical and physical structures.
This paper is devoted to the rigorous analysis of one such regime, namely the \emph{inertial limit}, in which inertial effects become negligible compared to pressure and viscous forces.

\medskip

We consider the following scaled compressible Navier--Stokes system posed on the $3$-dimensional torus $\mathbb{T}^3$:
\begin{equation}\label{eq:CNS-scaled}
\begin{cases}
&\partial_t\rho^\varepsilon + \Dv (\rho^\varepsilon u^\varepsilon)=0,\\
&
\varepsilon\partial_t(\rho^\varepsilon u^\varepsilon)+\varepsilon\Dv (\rho^\varepsilon u^\varepsilon\otimes u^\varepsilon)+\nabla p(\rho^\varepsilon)
= \nu\Delta u^\varepsilon + (\nu+\lambda)\nabla(\Dv u^\varepsilon),
\end{cases}
\end{equation}
where $\rho^\varepsilon=\rho^\varepsilon(t,x)\ge0$ denotes the density and $u^\varepsilon=u^\varepsilon(t,x)\in\mathbb{R}^3$ the velocity field.
The constants $\nu>0$ and $\nu+\lambda\ge0$ represent viscosity coefficients, and the pressure is assumed to be isentropic of the form
\[
p(\rho)=\rho^\gamma,\qquad \gamma>1.
\]

\medskip
The small parameter $\varepsilon>0$ measures the strength of inertia relative to pressure and viscous effects.
Formally, the scaling corresponds to a regime of \emph{strong friction} or \emph{high viscosity}, in which momentum relaxes rapidly and acceleration becomes negligible.
Such limits arise naturally, for instance, in highly viscous flows, porous media models, or slow dynamics dominated by internal stresses rather than transport.

\medskip

The system is supplemented with initial data
\begin{equation}\label{initial data}
\rho^\varepsilon(0,x)=\rho_0^\varepsilon(x),\qquad
(\rho^\varepsilon u^\varepsilon)(0,x)=m_0^\varepsilon(x).
\end{equation}
Throughout the paper we investigate the asymptotic behavior of finite-energy weak solutions as $\varepsilon\to0$. In short, we prove that finite-energy weak solutions converge to a pressure–viscosity dominated system, with complete loss of inertia.

\medskip

Formally neglecting the inertial terms in the momentum equation, the limit $\varepsilon\to0$ leads to the reduced system
\begin{equation}
\label{system of limit}
\begin{cases}
\partial_t \rho + \Dv (\rho u) = 0, \\[6pt]
\nabla p(\rho) = \nu\Delta u + (\nu + \lambda)\nabla(\Dv  u),
\end{cases}
\end{equation}
supplemented with the initial condition
\begin{equation}
\label{initial data for limit system}
\rho(0,\cdot)=\rho_0(\cdot)\quad\text{in }\mathbb{T}^3.
\end{equation}

\medskip

The limiting system exhibits a fundamental structural change.
While the original Navier-Stokes equations is a coupled system with hyperbolic transport with parabolic dissipation, the limit system is \emph{elliptic in the velocity} at each fixed time.
The velocity field is instantaneously slaved to the density through a pressure--viscosity balance, reflecting a complete loss of inertia.
As a result, the evolution becomes overdamped: momentum no longer propagates dynamically, but rather adjusts quasi-statically to the evolving density field. For the physical background, we refer the readers to Lions \cite{Lions} (see Chapter~7, Vol.~2).
From a modeling perspective, the limiting pressure–viscosity balance can be viewed as a Brinkman-type relation \cite{BellaOschmann2022,HNO}, linking the inertial limit of compressible flows to classical porous-media and filtration models.

\medskip

From a physical viewpoint, this regime corresponds to a situation in which kinetic energy cannot be sustained.
Although the velocity field need not vanish pointwise, its contribution to the total energy becomes negligible at the considered scaling.
One of the central objectives of this work is to show that this intuition can be made fully rigorous, even in the presence of vacuum and weak solutions.

\medskip

\subsection{Definitions}
We allow for localized regions of low density, including vacuum, and work within the framework of finite-energy weak solutions in the sense of Lions \cite{Lions}, Feireisl--Novotný--Petzeltová \cite{FNP}, and Feireisl \cite{F04}.  These results are analogous in 2D and 3D, except for the admissible range of the adiabatic exponent $\gamma$. Therefore, from this point onward, we formulate and analyze the problem in 3D. The main result and definitions of this paper remain valid in 2D setting as well.

\begin{Definition}
\label{weak solutions of CNS}
A pair $(\rho_\varepsilon, u_\varepsilon)$ is called a \emph{finite-energy weak solution} of \eqref{eq:CNS-scaled}-\eqref{initial data} on $[0,T]\times \mathbb{T}^3$ if the following conditions hold:

\subsection*{1. Regularity}
\[
\rho_\varepsilon \ge 0, \quad \rho_\varepsilon \in L^\infty(0,T; L^\gamma(\mathbb{T}^3)) \quad \text{for some } \gamma>1,
\]
\[
\rho_\varepsilon u_\varepsilon \in L^\infty(0,T; L^{\frac{2\gamma}{\gamma+1}}(\mathbb{T}^3)), \quad 
u_\varepsilon \in L^2(0,T; H^1(\mathbb{T}^3; \mathbb{R}^3)).
\]

\subsection*{2. Continuity Equation holds in the renormalized senese}

For every test function $\phi \in C_c^\infty([0,T)\times \mathbb{T}^3)$,
\begin{equation*}
\begin{split}
\int_0^T \int_{\mathbb{T}^3} \Big( b(\rho_\varepsilon) \partial_t \phi + b(\rho_\varepsilon) u_\varepsilon \cdot \nabla \phi \Big) \, dx\, dt
&+\int_{\mathbb{T}^3}(b'(\rho_{\varepsilon})\rho_{\varepsilon}-b(\rho_{\varepsilon}))  \Dv u_{\varepsilon}\phi\,dx
\\&+ \int_{\mathbb{T}^3} b(\rho_0^\varepsilon(x)) \phi(0,x)\, dx = 0.
\end{split}
\end{equation*}
This is equivalent to
\[
\partial_t b(\rho)
+ \Dv \bigl(b(\rho)u\bigr)
+ \bigl(b'(\rho)\rho - b(\rho)\bigr)\Dv u = 0
\]
in the sense of distributions on $(0,T)\times\mathbb{T}^3$,
for any $b \in C^{1}(\mathbb{R})$ such that
\[
b'(z) \equiv 0 \quad \text{for all } z \ge M,
\]
where the constant $M$ may depend on the choice of $b$.

\medskip

\subsection*{3. Weak Formulation of the Momentum Equation}

For all vector-valued test functions $\psi \in C_c^\infty([0,T)\times \mathbb{T}^3; \mathbb{R}^3)$,
\begin{align*}
&\varepsilon \int_0^T \int_{\mathbb{T}^3}  \Big( \rho_\varepsilon u_\varepsilon \cdot \partial_t \psi + \rho_\varepsilon u_\varepsilon \otimes u_\varepsilon : \nabla \psi \Big) \, dx\, dt
+ \int_0^T \int_{\mathbb{T}^3} \rho_\varepsilon^{\gamma} \, \Dv \psi \, dx\, dt \\
&= - \int_0^T \int_{\mathbb{T}^3} \Big( \nu \nabla u_\varepsilon : \nabla \psi + (\nu + \lambda)(\Dv u_\varepsilon)(\Dv \psi) \Big) \, dx\, dt 
+ \varepsilon \int_{\mathbb{T}^3} m_0^\varepsilon(x) \cdot \psi(0,x) \, dx.
\end{align*}

\subsection*{4. Energy Inequality (Finite Energy)}

For a.e. $t \in [0,T]$,
\begin{equation}
\begin{split}
\int_{\mathbb{T}^3} \Big( \frac{\varepsilon}{2} \rho_\varepsilon |u_\varepsilon|^2 + \frac{\rho_{\varepsilon}^{\gamma}}{\gamma-1} \Big) dx
&+ \int_0^t \int_{\mathbb{T}^3} \Big( \nu |\nabla u_\varepsilon|^2 + (\nu+\lambda) |\Dv u_\varepsilon|^2 \Big) dx\, ds
\\&
\le \int_{\mathbb{T}^3} \Big( \frac{\varepsilon}{2} \rho_0^\varepsilon |u_0^\varepsilon|^2 + \frac{(\rho_0^{\varepsilon})^{\gamma}}{\gamma-1} \Big) dx.
\end{split}
\end{equation}
\end{Definition}

Below $\mathcal{D}'$ denotes distributions in space-time and $C_c^\infty$ test functions are compactly supported.

\begin{Definition}\label{def weak solution for limit system}
A pair \((\rho,u)\) is called a \emph{weak solution} of \eqref{system of limit} and \eqref{initial data for limit system} on \([0,T]\times\mathbb{T}^3\)  if the following hold.

\medskip\noindent\textbf{(Regularity).} The functions satisfy, for some admissible exponent \(\gamma>1\),
\[
\rho\in L^\infty(0,T;L^\gamma(\mathbb{T}^3)),\qquad \rho\ge 0 \ \text{a.e.},
\]
and
\[
u\in L^2(0,T;H^1(\mathbb{T}^3;\mathbb{R}^3)).
\]

\medskip\noindent\textbf{{ Continuity Equation holds in the renormalized sense}}
For every test function $\varphi \in C_c^\infty([0,T)\times \mathbb{T}^3)$,
\begin{equation*}
\begin{split}
\int_0^T \int_{\mathbb{T}^3} \Big( b(\rho_\varepsilon) \partial_t \varphi + b(\rho_\varepsilon) u_\varepsilon \cdot \nabla \varphi \Big) \, dx\, dt
&+\int_{\mathbb{T}^3}(b'(\rho_{\varepsilon})\rho_{\varepsilon}-b(\rho_{\varepsilon}))  \Dv u_{\varepsilon}\varphi\,dx
\\&+ \int_{\mathbb{T}^3} b(\rho_0^\varepsilon(x)) \varphi(0,x)\, dx = 0.
\end{split}
\end{equation*}


\medskip\noindent\textbf{(Weak momentum / elliptic relation).} For almost every \(t\in(0,T)\) the pair \((\rho(t,\cdot),u(t,\cdot))\) satisfies the elliptic type equation in weak form: for every vector-valued test function \(\psi\in C_c^\infty(\mathbb{T}^3;\mathbb{R}^3)\),
\begin{equation}\label{eq:weak-mom}
\begin{split}
\int_{\mathbb{T}^3} \rho^{\gamma}\,\Dv \psi(x)\,dx
&+ \nu\int_{\mathbb{T}^3}\nabla u(t,x)\,:\,\nabla\psi(x)\,dx
\\&+ (\nu+\lambda)\int_{\mathbb{T}^3}(\Dv u(t,x))(\Dv \psi(x))\,dx
=0.
\end{split}
\end{equation}
Here \(\nabla u : \nabla\psi = \sum_{i,j}\partial_i u_j\,\partial_i\psi_j\).

\end{Definition}

\medskip

\subsection{Main Result}
Before stating our main results, we recall that the objects we study are well-posed in the weak sense. 
The existence of global weak solutions was established in \cite{FNP} for any $\gamma>\frac{3}{2}$  in3D and for any $\gamma>1$ in 2D.
In the following part of this paper, we need this restriction for $\gamma$.
This follows from the foundational work of Lions \cite{Lions} and the further development 
by Feireisl--Novotný--Petzeltová \cite{FNP} and Feireisl \cite{F04}, which also allow for the presence of vacuum. These results guarantee the existence of pairs $(\rho_\varepsilon, u_\varepsilon)$ 
satisfying the regularity, renormalized continuity, momentum, and energy inequalities stated in Definition~\ref{weak solutions of CNS}.

\medskip

For the limiting system \eqref{system of limit}, which reduces the momentum equation to a stationary elliptic relation for the velocity, 
existence of weak solutions is also established in Lions \cite{Lions} (see Chapter~7, Vol.~2).
These existence results provide a rigorous foundation for our subsequent analysis of the inertial limit, 
as they guarantee that the sequences of weak solutions considered in the limit process are indeed meaningful.

\medskip

While there is a substantial body of work \cite{BDGL,BFH,CVWY,Feireisl-incom,FKM,FKMV,Fu,KM,KM2,LM,MS0,MK,Ukai} on singular limits and related asymptotic regimes for compressible fluids, rigorous results on the inertial (small-mass / overdamped) limit for the compressible Navier–Stokes equations appear to be scarce. This highlights the novelty and significance of the current paper.

\medskip

After introducing precise notions of weak solutions for both the scaled system and the limiting system, we establish compactness and convergence of global weak solutions as $\varepsilon\to0$, under the well-preparedness assumption
\begin{equation}\label{intial energy limit}
\varepsilon\int_{\mathbb{T}^3} \rho_0^\varepsilon \lvert u_0^\varepsilon\rvert^2\, dx \to 0 \qquad \text{as } \varepsilon\to0.
\end{equation}
This condition rules out the persistence of an initial inertial layer in the limit. Throughout the paper, we assume that \eqref{intial energy limit} holds.

\begin{Theorem} [Inertial limit of finite-energy weak solutions] 
\label{main result}
For any $\gamma>\frac{3}{2}$, let 
$(\rho_\varepsilon, u_\varepsilon)$  be  a \emph{finite-energy weak solution} of \eqref{eq:CNS-scaled}-\eqref{initial data} on $[0,T]\times \mathbb{T}^3$.
Then there exists a subsequence (still denoted by \(\varepsilon\to0\)) such that

\begin{equation*}
\begin{split}
&\rho_{\varepsilon}\to \rho\quad\text{ weakly in }L^{\infty}(0,T;L^{\gamma}(\mathbb{T}^3)),  \text{ and } \rho_{\varepsilon}\to \rho \text{ in } L^1(0,T;L^1(\mathbb{T}^3)),
\\& u_\varepsilon\to u\quad\text{ weakly in } L^2(0,T;H^1(\mathbb{T}^3));
\end{split}
\end{equation*}
and the weak limit $(\rho, u)$ is a weak solution  of \eqref{system of limit} and \eqref{initial data for limit system} on \([0,T]\times\mathbb{T}^3\), 
and the weak solution satisfies the following energy equality
$$\int_{\mathbb{T}^3}\frac{\rho^{\gamma}}{\gamma-1}\,dx+\int_0^T\int_{\mathbb{T}^3}\nu |\nabla u|^2+(\nu+\lambda)|\Dv u|^2\,dx\,dt=\int_{\mathbb{T}^3}\frac{\rho_0^{\gamma}}{\gamma-1}\,dx$$
Moreover, the scaled kinetic energy vanishes in the limit: for every fixed $t>0$,
\[
\varepsilon \int_{\mathbb{T}^3} \rho^\varepsilon(t,x)\, \lvert u^\varepsilon(t,x)\rvert^2 \, dx
\longrightarrow 0
\quad \text{as } \varepsilon \to 0.
\]

\end{Theorem}

\begin{Remark}
Note that the limiting velocity field is determined instantaneously by the density 
through an elliptic relation. In particular, the limiting momentum equation does not 
generate an independent time evolution for the velocity, reflecting the overdamped 
nature of the limit dynamics.
\end{Remark}

Our main result shows that, under a natural assumption on finite initial kinetic energy,
finite-energy weak solutions of the scaled compressible Navier--Stokes system converge (up to a subsequence) to weak solutions of the reduced pressure--viscous system.
Moreover, we prove that the scaled kinetic energy vanishes in the limit and that the limiting solution satisfies an \emph{exact energy equality}.
This contrasts sharply with other singular limits, such as the vanishing viscosity limit, where anomalous dissipation may persist.

\medskip

The paper is organized as follows.
In Section~2, we establish compactness and stability properties of global weak solutions to the scaled
compressible Navier--Stokes system, relying on renormalization techniques in the sense of Lions~\cite{Lions}.
In Section~3, we prove that weak solutions of the limiting system satisfy an exact energy balance.
Finally, in Section~4, we complete the proof of the main convergence result and, in particular,
show that the scaled kinetic energy vanishes as $\varepsilon \to 0$.

\medskip

\section{Weak compactness as $\varepsilon\to0$}
The goal of this section is to prove the weak compactness of global weak solutions
to the scaled compressible Navier--Stokes equations as $\varepsilon \to 0$.

By the classical existence theories \cite{Lions, FNP, F04}, for any fixed
$\varepsilon>0$ and any adiabatic exponent $\gamma>\frac{3}{2}$, there exists
a global-in-time finite-energy weak solution
$(\rho_{\varepsilon},u_{\varepsilon})$ to the scaled compressible
Navier--Stokes system \eqref{eq:CNS-scaled}--\eqref{initial data} in 3D. 
These solutions satisfy the standard energy inequality, which provides
uniform-in-$\varepsilon$ bounds on the density and velocity. In particular, such solutions satisfy the following energy inequality: for a.e.
$t\ge 0$,
\begin{equation}
\begin{split}
\label{energy inequality for CNS}
\int_{\mathbb{T}^3} \Big(
\frac{\varepsilon}{2}\rho_\varepsilon |u_\varepsilon|^2
+ \frac{\rho_{\varepsilon}^{\gamma}}{\gamma-1}
\Big)\,dx
&+ \int_0^t \int_{\mathbb{T}^3}
\Big(
\nu |\nabla u_\varepsilon|^2
+ (\nu+\lambda) |\Dv u_\varepsilon|^2
\Big)\,dx\,ds
\\
&\le
\int_{\mathbb{T}^3} \Big(
\frac{\varepsilon}{2}\rho_0^\varepsilon |u_0^\varepsilon|^2
+ \frac{(\rho_0^\varepsilon)^{\gamma}}{\gamma-1}
\Big)\,dx .
\end{split}
\end{equation}

Once the right-hand side of the above energy inequality is finite, we immediately
deduce the following uniform bounds:
\begin{equation}
\label{uniform bounds from energy1}
\sqrt{\varepsilon}\,\sqrt{\rho_{\varepsilon}}\,u_{\varepsilon}
\in L^{\infty}\big(0,T;L^2(\mathbb{T}^3)\big),
\end{equation}
\begin{equation}
\label{uniform bounds from energy2}
\nabla u_{\varepsilon}
\in L^2\big(0,T;L^2(\mathbb{T}^3)\big),
\end{equation}
and
\begin{equation}
\label{uniform bounds from energy3}
\rho_{\varepsilon}
\in L^{\infty}\big(0,T;L^{\gamma}(\mathbb{T}^3)\big).
\end{equation}

By the definition of weak solutions, the pair
$(\rho_{\varepsilon},u_{\varepsilon})$ satisfies the continuity equation
in the sense of renormalized solutions.
Following \cite{FNP}, one can regularize the renormalized equation to obtain
\begin{equation}
\label{renormalized regularized}
\partial_t S_m\big(b(\rho_{\varepsilon})\big)
+ \Dv\!\left(
S_m\big(b(\rho_{\varepsilon})u_{\varepsilon}\big)
+ S_m\big(b'(\rho_{\varepsilon})\rho_{\varepsilon}
- b(\rho_{\varepsilon})\big)\Dv u_{\varepsilon}
\right)
= r_m ,
\end{equation}
where the remainder term satisfies
\[
r_m \to 0 \quad \text{in } L^2\big(0,T;L^2(\mathbb{T}^3)\big)
\quad \text{as } m\to\infty,
\]
provided that $b$ is a uniformly bounded $C^1$ function. Here $S_m(f)=f*\eta_m$ denotes the standard spatial mollification operator.

Consequently, by adapting the argument of Section~4.1 in \cite{FNP} and using
the uniform bounds
\eqref{uniform bounds from energy1}--\eqref{uniform bounds from energy3},
we obtain the following lemma.

\begin{Lemma}
\label{effective 1}
For any $\gamma>\frac{3}{2}$, then there exists a positive number $\theta$, only depending on $\gamma$, such that 
$$\int_0^T\int_{\mathbb{T}^3}\rho_{\varepsilon}^{\gamma+\theta}\,dx\,dt\leq C,$$
where $C$ is a constant.
\end{Lemma}

As a consequence of the previous lemma and basic energy inequality, which provides uniform energy bounds
and compactness properties for the density through the renormalized continuity
equation, we can establish the following proposition concerning the convergence
and compactness of the weak solutions as $\varepsilon \to 0$.

\begin{Proposition}
\label{prop of weak solution}
For any $\varepsilon>0$, there exists a subsequence of $(\rho_{\varepsilon},u_{\varepsilon})$ such that 
\begin{equation*}
\begin{split}
\label{convergence from energy}
&\rho_{\varepsilon}\to \rho \text{ in } C(0,T;L^{\gamma}_{weak}(\mathbb{T}^3)),\;\;\; \rho_{\varepsilon}\to \rho \text{ in } L^1(0,T;L^1(\mathbb{T}^3)),
\\&u_{\varepsilon}\to u \text{ weakly in } L^2(0,T;W^{1,2}(\mathbb{T}^3)),
\end{split}
\end{equation*}
and the weak limit $(\rho,u)$ is a weak solution to \eqref{system of limit}-\eqref{initial data for limit system}.
\end{Proposition}

\begin{proof}
With uniform bounds \eqref{uniform bounds from energy1}, \eqref{uniform bounds from energy2} and \eqref{uniform bounds from energy3}, and Lemma \ref {effective 1} at hands, we deduce that 
\begin{equation}
\begin{split}
\label{convergence from energy}
&\rho_{\varepsilon}\to \rho \text{ in } C(0,T;L^{\gamma}_{weak}(\mathbb{T}^3)),
\\&u_{\varepsilon}\to u \text{ weakly in } L^2(0,T;W^{1,2}(\mathbb{T}^3)),
\end{split}
\end{equation}

By Lemma \ref{effective 1},  we have 
\begin{equation}
\label{convergence of pressure}
\rho_{\varepsilon}^{\gamma}\to \bar{\rho^{\gamma}} \text{ weakly in } L^{\frac{\gamma+\theta}{\gamma}}((0,T)\times\mathbb{T}^3),
\end{equation}
where $\bar{f}$ is the weak limit of the sequence $\{f_n\}$.

For any $\gamma>\frac{3}{2}$,  and \eqref{convergence from energy} yields  
$$\rho_{\varepsilon}u_{\varepsilon}\to \rho u\text{ in}\;\;  D'((0,T)\times\mathbb{T}^3).$$

Note that 
\begin{equation*}
\begin{split}
\|\varepsilon\rho_{\varepsilon}u_{\varepsilon}\|_{L^{\infty}(0,T;L^{\frac{2\gamma}{\gamma+1}}(\mathbb{T}^3))}&\leq \sqrt{\varepsilon}\|\sqrt{\varepsilon\rho_{\varepsilon}}u_{\varepsilon}\|_{L^{\infty}(0,T;L^2(\mathbb{T}^3))}\|\sqrt{\rho_{\varepsilon}}\|_{L^{\infty}(0,T;L^{2\gamma}(\mathbb{T}^3))}
\\&\to 0 \quad\text{ as } \varepsilon \to 0,
\end{split}
\end{equation*}
thus we have $\varepsilon \rho_{\varepsilon}u_{\varepsilon}\to 0\text{ in } L^{\infty}(0,T;L^{\frac{2\gamma}{\gamma+1}}(\mathbb{T}^3)).$

With the help of  \eqref{convergence from energy},  this can ensure that 
$$\varepsilon\rho_{\varepsilon}u_{\varepsilon}\otimes u_{\varepsilon}\to 0 \text{ in }\;D'((0,T)\times\mathbb{T}^3)$$
for any $\gamma> \frac{3}{2}$.

Thus,  letting $\varepsilon\to 0$, we are able to recover the weak limit  $(\rho, u)$ of the sequence of solutions $(\rho_{\varepsilon},u_{\varepsilon})$. In particular, this weak limit  satisfies 
 \begin{equation}
\label{weak limit of system}
\begin{cases}
\partial_t \rho + \Dv(\rho u) = 0, \\[6pt]
\nabla \bar{\rho^{\gamma}} = \nu\Delta u + (\nu + \lambda)\nabla(\Dv u),
\end{cases}
\end{equation}
in the sense of $D'((0,T)\times\mathbb{T}^3).$  To show this weak limit is a weak solution to the equation, 
we still need to show that 
\[
\rho^{\gamma} = \overline{\rho^{\gamma}} .
\]
This follows from the strong convergence of the density in $L^1(0,T;L^1(\mathbb{T}^3))$.
The strong convergence of the density relies on uniform bounds and the weak continuity of the effective flux.
The proof closely follows the argument in Section~4 of \cite{FNP}.
Meanwhile, we can deduce the following energy inequality
\begin{equation}
\label{energy inequality for limit system}
\int_{\mathbb{T}^3}\frac{\rho^{\gamma}}{\gamma-1}\,dx+\int_0^T\int_{\mathbb{T}^3}\nu |\nabla u|^2+(\nu+\lambda)|\Dv u|^2\,dx\,dt\leq\int_{\mathbb{T}^3}\frac{\rho_0^{\gamma}}{\gamma-1}\,dx.
\end{equation}
\end{proof}

\section{Energy equality for the limit system}

In this section, we aim to upgrade the basic energy inequality
\eqref{energy inequality for limit system} to a full energy equality for all
weak solutions under consideration. To achieve this, we first require the
following lemma, which provides a crucial estimate allowing us to control the
pressure term  with respect to  the velocity in appropriate $L^p$ spaces.

\begin{Lemma}
\label{pressure bound}
Let $(\rho,u)$ be any weak solution to \eqref{system of limit}--\eqref{initial data for limit system}
satisfying the energy inequality \eqref{energy inequality for limit system}. Then
\[
\int_0^T \int_{\mathbb{T}^3} |\rho^{\gamma}|^2 \, dx \, dt \le C,
\]
where $C>0$ is a constant depending only on the initial data and the time $T>0$.
\end{Lemma}

  \begin{proof}
  We choose
\[
\varphi(t,x)
=  S_m\!\left[ \phi_m(t)\nabla \Delta^{-1}\!\left( S_m(\rho^{\gamma})
- \fint_{\mathbb{T}^3} S_m(\rho^{\gamma}) \right) \right]
\]
as a test function for any weak solution $(\rho,u)$ to the system
\eqref{system of limit}--\eqref{initial data for limit system}, where
$S_m(f)=f*\eta_m$ denotes the standard spatial mollification operator;
and $\phi_m \in \mathcal{D}(0,T)$ is a time cut-off function satisfying
$\phi_m \to 1$ as $m \to \infty$.

We remark that the operator $\nabla\Delta^{-1}$ coincides with the
Bogovskii operator on the torus, up to the subtraction of the spatial mean.
More generally, let $\Omega \subset \mathbb{R}^3$ be a bounded Lipschitz domain.
The \emph{Bogovskii operator}
\[
\mathcal{B} : L_0^p(\Omega) \longrightarrow W_0^{1,p}(\Omega;\mathbb{R}^3),
\qquad 1<p<\infty,
\]
is a linear operator associated with the divergence equation. Specifically,
for any
\[
f \in L_0^p(\Omega)
:= \Big\{ f \in L^p(\Omega) : \int_\Omega f\,dx = 0 \Big\},
\]
the vector field $u=\mathcal{B}[f]$ satisfies
\[
\Dv u = f \quad \text{in } \Omega,
\qquad
u|_{\partial \Omega} = 0.
\]

In the present proof, in order to control the terms $I_2$ and $I_3$ below, we rely
on the following standard properties of the Bogovskii operator:

\begin{enumerate}
\item \textbf{Boundedness.}
There exists a constant $C=C(\Omega,p)>0$ such that
\[
\|\nabla \mathcal{B}[f]\|_{L^p(\Omega)}
\le C \|f\|_{L^p(\Omega)}.
\]

\item \textbf{Divergence property.}
By construction, the operator satisfies
\[
\Dv \mathcal{B}[f] = f \quad \text{a.e. in } \Omega.
\]
\end{enumerate}

Testing \eqref{system of limit} with this function, we are able to derive that
  \begin{equation}
  \label{derive of control for pressure}
  \begin{split}
 & \int_0^T\int_{\mathbb{T}^3}\phi_m(t)|S_m(\rho^{\gamma})|^2\,dx\,dt =   \int_0^T\int_{\mathbb{T}^3}\phi_m(t) S_m(\rho^{\gamma}) \fint_{\mathbb{T}^3}S_m(\rho^{\gamma})\,dx\,dx\,dt
  \\&+\nu\int_0^T\int_{\mathbb{T}^3}\phi_m(t)\nabla S_m(u) (\partial_i\partial_j\Delta^{-1})\left(S_m(\rho^{\gamma})-\fint_{\mathbb{T}^3}S_m(\rho^{\gamma})\right)\,dx\,dt
\\&  +(\nu+\lambda)\int_0^T\int_{\mathbb{T}^3}\phi_m(t)\Dv S_m(u)\left(S_m(\rho^{\gamma})-\fint_{\mathbb{T}^3}S_m(\rho^{\gamma})\right)\,dx\,dt
  \\&=I_1+I_2+I_3,
    \end{split}
  \end{equation}where we  used \textbf{Divergence property} of Bogovskii operator to obtain
 $I_3$.
  
  Using \eqref{energy inequality for limit system}, we can control $|I_1|\leq C_1$, where $C_1$ is a constant only depending on time $T$ and the initial data.

To control $I_2$, we rely on \textbf{Boundedness} property of  the Bogovskii operator,
  this gives us 
  $$|I_2|\leq C_2\int _0^T\int_{\mathbb{T}^3}|\nabla S_m(u)|^2\,dx\,dt+\frac{1}{8}\int_0^T\int_{\mathbb{T}^3}\phi_m(t)|S_m(\rho^{\gamma})|^2\,dx\,dt.$$
 
    We are able to control $I_3$ as similarly
 $$|I_3|\leq C_3\int _0^T\int_{\mathbb{T}^3}|\Dv S_m(u)|^2\,dx\,dt+\frac{1}{8}\int_0^T\int_{\mathbb{T}^3}\phi_m(t)|S_m(\rho^{\gamma})|^2\,dx\,dt.$$
Using these controls of $I_1$, $I_2$ and $I_3$ to \eqref{derive of control for pressure},  this yields 
\begin{equation*}
\begin{split}
\int_0^T\int_{\mathbb{T}^3}\phi_m(t)|S_m(\rho^{\gamma})|^2\,dx\,dt &\leq C_1+C_2\int _0^T\int_{\mathbb{T}^3}|\nabla S_m(u)|^2\,dx\,dt
\\&+C_3\int _0^T\int_{\mathbb{T}^3}|\Dv S_m(u)|^2\,dx\,dt
\end{split}
\end{equation*} for all $m>0$. 

Thus, letting $m \to \infty$ in the above inequality, we obtain
\begin{equation*}
\begin{aligned}
\int_0^T \int_{\mathbb{T}^3} |\rho^{\gamma}|^2 \, dx \, dt
&\le C_1
+ C_2 \int_0^T \int_{\mathbb{T}^3} |\nabla u|^2 \, dx \, dt
+ C_3 \int_0^T \int_{\mathbb{T}^3} |\Dv u|^2 \, dx \, dt .
\end{aligned}
\end{equation*}
Observe that the right-hand side of the above inequality can be controlled by the basic energy inequality \eqref{energy inequality for limit system}. Consequently, we obtain the bound of the pressure $\rho^{\gamma}$.
  \end{proof}
  
Once an $L^{p}$-bound for the pressure term has been established in Lemma \ref{pressure bound}, we are in a position to upgrade the basic energy inequality to an energy equality. More precisely, we prove the following result.

  \begin{Proposition}
  \label{Propo for energy}
Let $(\rho,u)$ be any weak solution to \eqref{system of limit}--\eqref{initial data for limit system}
satisfying the energy inequality \eqref{energy inequality for limit system}. Then the following energy equality
\begin{equation}
\label{energy equality}
\int_{\mathbb{T}^3}\frac{\rho^{\gamma}}{\gamma-1}\,dx+\int_0^T\int_{\mathbb{T}^3}\nu |\nabla u|^2+(\nu+\lambda)|\Dv u|^2\,dx\,dt=\int_{\mathbb{T}^3}\frac{\rho_0^{\gamma}}{\gamma-1}\,dx
\end{equation}
holds.
  \end{Proposition}

  \begin{proof}
 Note that the weak limit $(\rho,u)$ is a renormalized solution. That is,
\[
\partial_t b(\rho)+\operatorname{div}(b(\rho)u)
+\big(b'(\rho)\rho-b(\rho)\big)\operatorname{div}u=0
\]
holds in the sense of $\mathcal{D}'((0,T)\times\mathbb{T}^3)$ for any
$b\in C^{1}(\mathbb{R})$.

We regularize the above equation to obtain
\begin{equation}
\label{renormalized regu}
\partial_t S_m(b(\rho))
+\operatorname{div}\big(S_m(b(\rho)u)\big)
+S_m\!\left(b'(\rho)\rho-b(\rho)\right) \Dv(S_m(u))
=R_m,
\end{equation}
where the commutator term $R_m$ is given by
\[
R_m
:=S_m\!\left(b'(\rho)\rho-b(\rho)\right) \Dv (S_m(u))
- S_m\!\left(\big(b'(\rho)\rho-b(\rho)\big)\operatorname{div}u\right).
\]

To derive the energy equality, we choose $S_m(S_m(u))$ as a test function to derive the following one
\begin{equation}
\label{derivation of energy equality}
\begin{split}
\int_0^T\int_{\mathbb{T}^3}S_m(\rho^\gamma)\Dv (S_m(u))\,dx\,dt
&=-\nu \int_0^T\int_{\mathbb{T}^3}|\nabla S_m(u)|^2\,dx\,dt
\\&-(\nu+\lambda)\int_0^T\int_{\mathbb{T}^3}|\Dv S_m(u)|^2\,dx\,dt.
\end{split}
\end{equation}
We take $b(\rho)=\rho^{\gamma}$ for \eqref{renormalized regu}, then 
$$\partial_tS_m(\rho^{\gamma})+\Dv(S_m(\rho^{\gamma})u)+(\gamma-1)S_m(\rho^{\gamma})\Dv S_m(u)=R_m.$$
Thus, 
the left hand side of \eqref{derivation of energy equality} will be given by 
$$\int_0^T\int_{\mathbb{T}^3}\frac{1}{\gamma-1}\partial_tS_m(\rho^{\gamma})\,dx\,dt+\int_0^T\int_{\mathbb{T}^3}R_m\,dx\,dt.$$
This gives us the following one
\begin{equation}
\label{smooth energy}
\begin{split}&
\int_{\mathbb{T}^3}\frac{1}{\gamma-1}S_m(\rho^{\gamma})\,dx\,dt+\int_0^T\int_{\mathbb{T}^3}R_m\,dx\,dt
\\&+\nu\int_0^T \int_{\mathbb{T}^3}|\nabla S_m(u)|^2\,dx\,dt+(\nu+\lambda)\int_0^T\int_{\mathbb{T}^3}|\Dv S_m(u)|^2\,dx\,dt
\\&=\int_{\mathbb{T}^3}\frac{1}{\gamma-1}S_m(\rho_0^{\gamma})\,dx
\end{split}
\end{equation}

Now let us to show that $R_m \to 0$ in $L^1((0,T)\times\mathbb T^d)$.
Recall that for $b(\rho)=\rho^\gamma$ we have
\begin{equation*}
\begin{split}&
R_m
:= S_m\!\big((b'(\rho)\rho-b(\rho))\big)\,\operatorname{div} S_m(u)
   - S_m\!\Big(\big(b'(\rho)\rho-b(\rho)\big)\operatorname{div} u\Big)
   \\&
= (\gamma-1)\Big(S_m(\rho^\gamma)\,\operatorname{div} S_m u - S_m(\rho^\gamma\,\operatorname{div} u)\Big).
\end{split}
\end{equation*}

This allows us to rewrite it as 
\[
R_m = (\gamma-1)\,\Big(
\underbrace{S_m(\rho^\gamma)\,\operatorname{div} S_m u - S_m(\rho^\gamma)\,\operatorname{div} u}_{=:A_m}
+
\underbrace{S_m(\rho^\gamma)\,\operatorname{div} u - S_m(\rho^\gamma\,\operatorname{div} u)}_{=:B_m}
\Big),
\]
thus we only need to handle $A_m$ and $B_m$ to control $R_m\to 0$ as $m$ goes to large.

\medskip
\noindent\emph{Estimate of $A_m$.}
By Hölder in $(t,x)$ with exponents $2$--$2$ and the strong convergence
$\operatorname{div} S_m u \to \operatorname{div} u$ in $L^2((0,T)\times\mathbb T^d)$,
\[
\|A_m\|_{L^1((0,T)\times\mathbb T^d)}
\le \|S_m(\rho^\gamma)\|_{L^2((0,T)\times\mathbb T^d)}\,
   \|\operatorname{div} S_m u - \operatorname{div} u\|_{L^2((0,T)\times\mathbb T^d)} \xrightarrow[m\to\infty]{} 0.
\]
Here we used that $S_m(\rho^\gamma)\to \rho^\gamma$ in $L^2$ and is uniformly bounded in
$L^2$ thanks to Lemma~3.1 (i.e. $\rho^\gamma\in L^2((0,T)\times\mathbb T^d)$).

\medskip
\noindent\emph{Estimate of $B_m$ (Friedrichs commutator).}
We introduce  $f:=\rho^\gamma$, and $g:=\operatorname{div} u$, then $$f\in L^2((0,T)\times\mathbb T^d)$$ by Lemma~3.1, 
and $$g\in L^2((0,T)\times\mathbb T^d)$$ since $u\in L^2(0,T;H^1(\mathbb T^d))$.
The mollifiers $(S_m)_{m\ge1}$ satisfy the standard product-commutator property:
for any $f\in L^p$, $g\in L^q$ with $\frac1p+\frac1q=1$,
\[
\|S_m(f)\,g - S_m(fg)\|_{L^1} \xrightarrow[m\to\infty]{} 0.
\]
Applying this with $p=q=2$ yields
\[
\|B_m\|_{L^1((0,T)\times\mathbb T^d)}
= \|S_m(f)\,g - S_m(fg)\|_{L^1((0,T)\times\mathbb T^d)}
\xrightarrow[m\to\infty]{} 0.
\]

\medskip
\noindent
Combining the estimates for $A_m$ and $B_m$, we conclude that
\[
\|R_m\|_{L^1((0,T)\times\mathbb T^d)} \xrightarrow[m\to\infty]{} 0,
\]
i.e. $R_m \to 0$ strongly in $L^1(0,T;L^1(\mathbb T^d))$.

Consequently, by passing to the limit \(m\to\infty\) in
\eqref{smooth energy}, we are able to recover the energy equality
\eqref{energy equality}.

\end{proof}


\section{Proof to the Main theorem}

 The goal of this section is to rigorously justify the inertial limit of weak solutions to the compressible Navier--Stokes equations. In view of Proposition~\ref{prop of weak solution}, by letting $\varepsilon\to0$ we obtain a weak solution to the limit system
\eqref{system of limit}--\eqref{initial data for limit system}, which satisfies
Definition~\ref{def weak solution for limit system}. 

A key difficulty in the limiting procedure is the possible loss of energy in the weak formulation. This issue is resolved by Proposition~\ref{Propo for energy}, which shows that any weak solution of the limit system obtained in this way satisfies the energy equality
\eqref{energy equality}. As a consequence, the inertial limit preserves the full energy balance of the system.

To prove our main result, we still need to show 
\begin{equation}
\label{eq:kinetic-vanish-fixed-t} 
\varepsilon\int_{\mathbb{T}^3}\rho_{\varepsilon}(x,t)|u_{\varepsilon}(x,t)|^{2}\,dx
\;\longrightarrow\; 0
\qquad\text{as }\varepsilon\to0,
\end{equation}
for any fixed $t\in[0,T].$
To this end, 
we argue by contradiction.
Assume that \eqref{eq:kinetic-vanish-fixed-t} is false for this fixed $t$.
Then there exist a constant $\varepsilon_0>0$ and a sequence $\varepsilon_k\to 0$ such that
\begin{equation}\label{eq:contradiction-lowerbound}
\varepsilon_k\int_{\mathbb T^d}\rho_{\varepsilon_k}(t,x)\,|u_{\varepsilon_k}(t,x)|^2\,dx
\ge \varepsilon_0
\qquad \text{for all }k.
\end{equation}

By the energy inequality for the compressible Navier--Stokes system, for any fixed t,  we have
for each \(\varepsilon_k\),
\begin{align*}
&\frac{\varepsilon_k}{2}\int_{\mathbb{T}^3}\rho_{\varepsilon_k}(x,t)|u_{\varepsilon_k}(x,t)|^{2}\,dx
+ \int_{\mathbb{T}^3}\frac{\rho_{\varepsilon_k}^{\gamma}(x,t)}{\gamma-1}\,dx 
\\&+ \int_0^t\int_{\mathbb{T}^3}
\nu|\nabla u_{\varepsilon_k}(x,s)|^2
+(\nu+\lambda)|\Dv u_{\varepsilon_k}(x,s)|^2\,dx\,ds
\\&\;\le\;\varepsilon_k\int_{\mathbb{T}^3}\frac{1}{2}\rho^0_{\varepsilon_k}|u^0_{\varepsilon_k}|^2\,dx+
\int_{\mathbb{T}^3}\frac{\rho_0^{\gamma}}{\gamma-1}\,dx .
\end{align*}

Using the assumed lower bound, this yields
\begin{align*}
\frac{\varepsilon_0}{2}
+ \int_{\mathbb{T}^3}\frac{\rho_{\varepsilon_k}^{\gamma}(x,t)}{\gamma-1}\,dx 
&+ \int_0^t\int_{\mathbb{T}^3}
\nu|\nabla u_{\varepsilon_k}(x,s)|^2
+(\nu+\lambda)|\Dv u_{\varepsilon_k}(x,s)|^2\,dx\,ds
\\&\;\le\;\varepsilon_k\int_{\mathbb{T}^3}\frac{1}{2}\rho^0_{\varepsilon_k}|u^0_{\varepsilon_k}|^2\,dx+
\int_{\mathbb{T}^3}\frac{\rho_0^{\gamma}}{\gamma-1}\,dx .
\end{align*}

Passing to the limit \(k\to\infty\), and using the convergences established
earlier together with the lower semicontinuity of the dissipation terms, we
obtain
\begin{align*}
\frac{\varepsilon_0}{2}
+ \int_{\mathbb{T}^3}\frac{\rho^{\gamma}(x,t)}{\gamma-1}\,dx &+ \int_0^t\int_{\mathbb{T}^3}
\nu|\nabla u(x,s)|^2\,dx\,dt
\\&+\int_0^t\int_{\mathbb{T}^3}
(\nu+\lambda)|\Dv u(x,s)|^2\,dx\,ds
\;\le\;
\int_{\mathbb{T}^3}\frac{\rho_0^{\gamma}}{\gamma-1}\,dx .
\end{align*}

Since \(\varepsilon_0>0\), this implies
\begin{align*}
 \int_{\mathbb{T}^3}\frac{\rho^{\gamma}(x,t)}{\gamma-1}\,dx &+ \int_0^t\int_{\mathbb{T}^3}
\nu|\nabla u(x,s)|^2\,dx\,dt
\\&+\int_0^t\int_{\mathbb{T}^3}
(\nu+\lambda)|\Dv u(x,s)|^2\,dx\,ds
<
\int_{\mathbb{T}^3}\frac{\rho_0^{\gamma}}{\gamma-1}\,dx, 
\end{align*}which contradicts Proposition \ref{Propo for energy}, that is, 
$(\rho,u)$  satisfies the energy equality \eqref{energy equality}.
\medskip

Consequently, we have established \eqref{eq:kinetic-vanish-fixed-t} for any fixed $t\in[0,T]$. This completes the proof of our main result.

  \section*{Acknowledgments}

The author  is partially supported
by the NSF grant: DMS-2510425, and by the Simons Foundation: MPS-TSM-00007824.


\bigskip\bigskip
\begin{thebibliography}{99}


\bibitem{BellaOschmann2022}
P.~Bella and F.~Oschmann,
\newblock Homogenization and low Mach number limit of compressible Navier--Stokes equations in critically perforated domains,
\newblock \emph{Journal of Mathematical Fluid Mechanics}, 24(3):79, 2022.



 \bibitem{BDGL}D. Bresch, B. Desjardins, E. Grenier, and C.-K. Lin, Low Mach number limit of viscous polytropic flows: formal
asymptotics in the periodic case, Stud. Appl. Math., 109 (2): 125--149, 2002.


\bibitem{BFH}D. Breit, E. Feireisl, M. Hofmanova, Incompressible limit for compressible fluids with stochastic forcing,
Arch. Ration. Mech. Anal. 222 (2016), no. 2, 895--926.



\bibitem{CVWY} M. Chen, A. Vasseur, D. Wang and C. Yu,  Universality in the Low Mach number limit via a convex integration framework, arXiv:2601.19744, 2026.


 \bibitem{Lions} P.-L. Lions, 
Mathematical topics in fluid mechanics. Vol. 2.
Compressible models, 
Oxford University Press, New York, 1998. xiv+348 pp.
ISBN:0-19-851488-3
\bibitem{FNP}
E. Feireisl, A.  Novotný, H.  Petzeltová,
 On the existence of globally defined weak solutions to the Navier-Stokes equations.
J. Math. Fluid Mech. 3 (2001), no. 4, 358–392.



\bibitem{F04} E. Feireisl, Dynamics of viscous compressible fluids.
Oxford University Press, Oxford, 2004. xii+212 pp.
ISBN:0-19-852838-8




\bibitem{Feireisl-incom} 
E. Feireisl, Incompressible limits and propagation of acoustic waves in large domains with boundaries, Comm. Math. Phys., 294(1):73--95, 2010.




\bibitem{FKM}
E. Feireisl, C. Klingenberg, S. Markfelder, On the low Mach number limit for the compressible Euler
system, SIAM J. Math. Anal. 51(2):1496–1513, 2019

\bibitem{FKMV} E. Feireisl, O. Kreml, V. Macha, S.Necasova, 
On the low Mach number limit of compressible flows in exterior moving domains, J. Evol. Equ. 16 (2016), no. 3, 705--722.


\bibitem{Fu}
M. Fujii,  Low Mach number limit of the global solution to the compressible Navier--Stokes system for large data in the critical Besov space, Math. Ann. 388 (2024), no. 4, 4083--4134.



\bibitem{HNO}
R.~M.~Höfer,  S.~Nečasová, and F.~Oschmann,
\newblock Quantitative homogenization of the compressible Navier--Stokes equations towards Darcy's law,
\newblock \emph{Annales de l'Institut Henri Poincaré C, Analyse Non Linéaire}, 
published online 23 May 2025.



\bibitem{KM}
S. Klainerman and A. Majda, Singular limits of quasilinear systems with large
parameter and the incompressible limit of compressible fluids, Comm. Pure Appl. Math. 34 (1981), 481-524.


\bibitem{KM2} S. Klainerman and A. Majda, Compressible and incompressible fluids, Comm.
Pure Appl. Math. 35 (1982), 637-656.


\bibitem{LM} P.-L. Lions and N. Masmoudi, Incompressible limit for a viscous compressible fluid, J. Math. Pures Appl. (9),
77(6):585--627, 1998.





\bibitem{MS0} G. Metivier and S. Schochet, The incompressible limit of the non-isentropic Euler equations, Arch. Ration. Mech. Anal., 158(1): 61--90, 2001.

\bibitem{MK} S. Markfelder and C. Klingenberg, The Riemann problem for the multidimensional isentropic system of gas dynamics is ill-posed if it contains a shock, Arch. Ration. Mech. Anal., 227:967--994, 2018.




\bibitem{Ukai}S. Ukai, The incompressible limit and the initial layer of the compressible Euler equation, J. Math. Kyoto Univ., 26(2): 323--331, 1986.




\end{thebibliography}
\end{document}